\documentclass[11pt,reqno]{amsart}

\usepackage[T1]{fontenc}
\usepackage[a4paper,margin=2.8cm]{geometry}

\usepackage{amsmath,amssymb,amsfonts,amsthm,mathtools,mathrsfs}
\usepackage{enumitem}
\usepackage{xcolor}

\usepackage{silence}
\usepackage{microtype}

\usepackage[hidelinks,bookmarksdepth=3]{hyperref}
\usepackage[nameinlink,noabbrev]{cleveref}

\usepackage[
  style=alphabetic,
  backend=biber,
  maxbibnames=99,
  maxcitenames=99
]{biblatex}

\AtBeginDocument{
  \setlength{\abovedisplayskip}{9pt plus 2pt minus 3pt}
  \setlength{\belowdisplayskip}{9pt plus 2pt minus 3pt}
  \setlength{\abovedisplayshortskip}{6pt plus 2pt minus 2pt}
  \setlength{\belowdisplayshortskip}{6pt plus 2pt minus 2pt}
  \setlength{\jot}{4pt}
}

\numberwithin{equation}{section}

\newcommand{\eps}{\varepsilon}

\newcommand{\E}{\mathbb E}
\newcommand{\1}{\mathbf 1}

\DeclarePairedDelimiter{\abs}{\lvert}{\rvert}

\theoremstyle{plain}
\newtheorem{thm}{Theorem}[section]
\newtheorem{cor}[thm]{Corollary}
\newtheorem{lem}[thm]{Lemma}

\theoremstyle{remark}

\title[A unified abstract regularity lemma]{A unified abstract regularity lemma}

\author[G. Carenini]{Gaia Carenini}
\address{Trinity College Cambridge, Department of Pure Mathematics and Mathematical Statistics, Centre for Mathematical Sciences, Wilberforce Road, Cambridge CB3 0WA, United Kingdom}
\email{gc645@cam.ac.uk}

\author[L. Franchi]{Leonardo Franchi}
\address{Department of Pure Mathematics and Mathematical Statistics, Centre for Mathematical Sciences, Wilberforce Road, Cambridge CB3 0WA, United Kingdom}
\email{lf511@cam.ac.uk}

\begin{document}
\begin{abstract}
The goal of this short note is to prove a unified abstract regularity lemma which recovers Szemerédi's graph
regularity lemma, Green's arithmetic regularity lemma, and a regularity lemma for Boolean functions as direct corollaries.
\end{abstract}
\maketitle
\vspace{-0.75cm}
\tableofcontents 
\vspace{-0.75cm}
\section{Introduction}

Regularity lemmas are among the basic structural tools of modern discrete
mathematics.  They assert, in different settings, that an arbitrary object can
be decomposed into a bounded-complexity structured part and a pseudorandom
remainder.  The paradigmatic example is Szemerédi's graph regularity lemma
\cite{Szemeredi1978}; closely related regularity statements appear in additive
combinatorics, most notably Green's arithmetic regularity lemma
\cite{Green2005}, and in the analysis of Boolean functions, for instance in
the regularity lemma used by Minzer \cite{minzer2021analysisbooleanfunctionslec1112} to recover the result of Dinur and Friedgut asserting that an intersecting family lies in a junta \cite{DinurFriedgut2009}. These regularity lemmas have found a remarkably broad range of applications across discrete mathematics, playing a central role in developments in graph theory, additive combinatorics, theoretical computer science, and the analysis of Boolean functions.

A useful distinction is between weak and strong regularity lemmas. 
Weak
regularity lemmas, such as the Frieze--Kannan lemma, admit a rather general
abstract formulation; one such formulation is due to Trevisan, Tulsiani and
Vadhan \cite{trevisan2009regularity}. They typically yield quantitatively
much stronger bounds at the expense of providing a correspondingly weaker
structural description. For strong regularity lemmas, the situation is less uniform.  There
are abstract approaches, including the probabilistic framework of Tao \cite{Tao2006} and the
semiring framework of Bollobás and Nikiforov \cite{BollobasNikiforov2007}, but these do not give, in a
completely direct way, all of the standard examples one would like to recover,
especially the Fourier-analytic form of Green's arithmetic regularity lemma.

The purpose of this short note is to isolate an elementary abstract statement from
which the most standard strong regularity lemmas follow with essentially no additional work. This
also clarifies why the same energy-increment proof is operating in graph, Fourier-analytic, and Boolean-function settings. For completeness, in an appendix, we also revisit the frameworks of Tao and of Bollobás--Nikiforov, showing how they can be recovered from our abstract formulation.

\medskip
\textbf{Organization.} The remainder of this note is organized as follows. In Section~\ref{sec: abs-reg-lem}, we state and prove our unified abstract regularity lemma (Theorem~\ref{thm:abstract-regularity}). In Section~\ref{sec:green}, we derive Green's arithmetic regularity lemma from Theorem~\ref{thm:abstract-regularity}. In Section~\ref{sec: Szemeredi}, we derive Szemerédi's graph regularity lemma from the same abstract theorem. In Section~\ref{sec: Dinur-Friedgut}, we obtain the aforementioned regularity lemma for Boolean functions. In the appendix, we show how our framework recovers the abstract regularity lemmas of Bollobás and Nikiforov, and of Tao.\\

\textbf{Acknowledgments.}
We are grateful to our supervisors, Imre Leader and Timothy Gowers, for their support and guidance; the first author is supported by the CB European PhD Studentship funded by Trinity College, Cambridge, and the second author acknowledges support from the Isaac Newton Trust through a Trinity Cambridge Research Studentship.

\section{The abstract regularity lemma }\label{sec: abs-reg-lem}
Let $(X,\mathcal A,\mu)$ be a normalized finitely additive measure space, meaning that $X$ is a set, $\mathcal A$ is an algebra of subsets of $X$, and $\mu:\mathcal A\to[0,1]$ is a finitely additive function satisfying $\mu(X)=1$. Given a set $P\in \mathcal{A}$ with $\mu(P)>0$, we denote the normalized average of a bounded, $\mathcal{A}$-measurable function $h:X\rightarrow\mathbb{C}$ over $P$ by
$$
\E_P[h]=\frac{1}{\mu(P)}\int_P h\,d\mu.
$$ 

In what follows, all partitions considered are finite partitions into measurable sets of positive measure, meaning a partition $\mathcal{P}$ of $X$ is a finite collection of disjoint sets $\{P_1, \dots, P_n\} \subset \mathcal{A}$ such that $\bigcup_{i=1}^n P_i = X$ and $\mu(P_i) > 0$ for all $i \in \{1, \dots, n\}$. Let $\mathcal{P}$ and $\mathcal{Q}$ be two partitions of $X$. We say that $\mathcal{Q}$ \textit{refines} $\mathcal{P}$ if, for every $Q \in \mathcal{Q}$, there exists some $P \in \mathcal{P}$ such that $Q \subseteq P$. Given a part $P \in \mathcal{P}$, the partition of $P$ induced by $\mathcal{Q}$, denoted by $\mathcal{Q}|_P$, is defined as the collection of sets $\mathcal{Q}|_P = \{Q \in \mathcal{Q} \mid Q \subseteq P\}$.

For every part $P \in \mathcal{P}$ of a partition $\mathcal{P}$, fix a non-empty family of measurable functions $\mathcal{F}_P \subseteq \{f : P \to \mathbb{C} \mid \|f\|_\infty \le 1\}$. Given an $\mathcal{A}$-measurable function $g: X \to [0, 1]$, we say that $\mathcal{P}$ is $\varepsilon$-\textit{regular} for $g$ with respect to the families $(\mathcal{F}_P)_{P\in\mathcal P}$ whenever the following inequality holds
$$
\sum_{P \in \mathcal{P}} \mu(P) \sup_{f \in \mathcal{F}_P} \left| \mathbb{E}_P \left[ (g|_P - \mathbb{E}_P[g|_P]) \cdot f \right] \right| \le \varepsilon,
$$
where $g|_P$ represents the restriction of $g$ to the part $P$. For notational simplicity, we define the irregularity of $g$ on $P$ as
$$ \operatorname{irr}_g(P) = \sup_{f \in \mathcal{F}_P} \left| \mathbb{E}_P \left[ (g|_P - \mathbb{E}_P[g|_P]) \cdot f \right] \right|.$$
   
For every positive integer $M$, we choose a subclass of the partitions of $X$ with at most $M$ parts. Its elements will be called admissible partitions, and the subclass will be denoted by $\mathsf{Part}_M(X)$. A \emph{refinement scheme} consists of a growth function $\Phi : \mathbb{N} \to \mathbb{N}$ and a refinement function $\psi$ that, given a partition $\mathcal{P} \in \mathsf{Part}_M(X)$ and a family of witnesses $\{f_P\}_{P \in \mathcal{P}}$ with $f_P \in \mathcal{F}_P$ for each $P \in \mathcal{P}$, returns a refinement $\mathcal{P}' = \psi(\mathcal{P}, \{f_P\}_{P \in \mathcal{P}}) \in \mathsf{Part}_{\Phi(M)}(X)$ of $\mathcal{P}$. 

In this work, we assume that the chosen witnesses become measurable with respect to the refinement, i.e., if $\mathcal{P}' = \psi(\mathcal{P}, \{f_P\}_{P \in \mathcal{P}})$, then for each $P \in \mathcal{P}$, the function $f_P$ is constant on every part of the restricted partition $\mathcal{P}'|_P$.

This is the distinctive feature of our framework. By isolating the measurability of the chosen witnesses with respect to the refinement as the essential hypothesis, one obtains an abstract regularity lemma that is at once simple, yet general.

With this notation in hand, we can state the unified abstract regularity lemma. 

\begin{thm}[Unified abstract regularity lemma]\label{thm:abstract-regularity}
Let $(X, \mathcal{A}, \mu)$ be a normalized finitely additive measure space. Let $g: X \to [0, 1]$ be an $\mathcal{A}$-measurable function, let $\varepsilon > 0$, and let $\mathcal{P}_0 \in \mathsf{Part}_{M_0}(X)$. Then there exists an integer $t \le \lceil \varepsilon^{-2} \rceil$ and a partition $\mathcal{P}_t \in \mathsf{Part}_{\Phi^{\circ t}(M_0)}(X)$ which is $\varepsilon$-regular for $g$, where $\Phi^{\circ t}$ denotes the $t$-fold composition of $\Phi$ with itself.
\end{thm}

\begin{proof}
For a partition $\mathcal{P}$, define the energy of $g$ by
$$
\mathcal{E}_g(\mathcal{P}) = \|\mathbb{E}[g \mid \mathcal{P}]\|_{L^2(\mu)}^2 = \sum_{P \in \mathcal{P}} \mu(P) \bigl(\mathbb{E}_P[g|_P]\bigr)^2.
$$

The energy always lies between $0$ and $1$, since $0 \le \mathcal{E}_g(\mathcal{P}) \le \mathbb{E}_\mu[g^2] \le 1$. We first record the energy increment estimate.

\smallskip
\noindent\emph{Claim.} Let $\mathcal{Q}$ be a partition that refines $\mathcal{P}$. Suppose that, for each $P \in \mathcal{P}$, the function $f_P : P \to \mathbb{C}$ satisfies $\|f_P\|_\infty \le 1$ and is constant on every part of $\mathcal{Q}|_P$. Then
$$
\mathcal{E}_g(\mathcal{Q}) - \mathcal{E}_g(\mathcal{P}) \ge \sum_{P \in \mathcal{P}} \mu(P) \left| \mathbb{E}_P \left[ (g|_P - \mathbb{E}_P[g|_P]) \cdot f_P \right] \right|^2.
$$

\smallskip
\noindent\emph{Proof of the claim.} Since $\mathcal{Q}$ refines $\mathcal{P}$, we can write
$$
\mathcal{E}_g(\mathcal{Q}) = \sum_{P \in \mathcal{P}} \mu(P) \|\mathbb{E}_P[g|_P \mid \mathcal{Q}|_P]\|_{L^2(P)}^2.
$$
For $P \in \mathcal{P}$, define
$h_P = \mathbb{E}_P[g|_P \mid \mathcal{Q}|_P] - \mathbb{E}_P[g|_P]$.
Then $\mathbb{E}_P[h_P] = 0$, and hence
$$
\|\mathbb{E}_P[g|_P \mid \mathcal{Q}|_P]\|_{L^2(P)}^2 = \bigl(\mathbb{E}_P[g|_P]\bigr)^2 + \|h_P\|_{L^2(P)}^2.
$$
Summing up, we obtain
$$
\mathcal{E}_g(\mathcal{Q}) - \mathcal{E}_g(\mathcal{P}) = \sum_{P \in \mathcal{P}} \mu(P) \|h_P\|_{L^2(P)}^2.
$$
Because $f_P$ is constant on the parts of $\mathcal{Q}|_P$, we have
$$
\mathbb{E}_P \left[ (g|_P - \mathbb{E}_P[g|_P]) \cdot f_P \right] = \mathbb{E}_P[h_P f_P].
$$
By the Cauchy--Schwarz inequality,
$$
\left|\mathbb{E}_P[h_P f_P]\right|^2
\le
\mathbb{E}_P[|h_P|^2]\,\mathbb{E}_P[|f_P|^2]
\le
\|h_P\|_{L^2(P)}^2.
$$
Multiplying by $\mu(P)$ and summing over all $P \in \mathcal{P}$ proves the claim. $\square$

\smallskip

We now run the energy increment procedure. Starting from $\mathcal{P}_0$, suppose that $\mathcal{P}_s$ has been constructed. If $\mathcal{P}_s$ is $\varepsilon$-regular for $g$, we stop. Otherwise,
$$
\sum_{P \in \mathcal{P}_s} \mu(P) \sup_{f \in \mathcal{F}_P} \left| \mathbb{E}_P \left[ (g|_P - \mathbb{E}_P[g|_P]) \cdot f \right] \right| > \varepsilon.
$$
Since $\mathcal{P}_s$ is finite, we may choose witnesses $f_P \in \mathcal{F}_P$ so that
$$
\sum_{P \in \mathcal{P}_s} \mu(P) \left| \mathbb{E}_P \left[ (g|_P - \mathbb{E}_P[g|_P]) \cdot f_P \right] \right| > \varepsilon.
$$
Set $\mathcal{P}_{s+1} = \psi(\mathcal{P}_s, \{f_P\}_{P \in \mathcal{P}_s})$. By the defining property of the refinement scheme, each $f_P$ is constant on the parts of $\mathcal{P}_{s+1}|_P$. The claim then gives
$$
\mathcal{E}_g(\mathcal{P}_{s+1}) - \mathcal{E}_g(\mathcal{P}_s) \ge \sum_{P \in \mathcal{P}_s} \mu(P) \left| \mathbb{E}_P \left[ (g|_P - \mathbb{E}_P[g|_P]) \cdot f_P \right] \right|^2.
$$
Since $\sum_{P \in \mathcal{P}_s} \mu(P) = 1$, the Cauchy-Schwarz inequality yields
$$
\sum_{P \in \mathcal{P}_s} \mu(P) \left| \mathbb{E}_P \left[ (g|_P - \mathbb{E}_P[g|_P]) \cdot f_P \right] \right|^2 \ge \left( \sum_{P \in \mathcal{P}_s} \mu(P) \left| \mathbb{E}_P \left[ (g|_P - \mathbb{E}_P[g|_P]) \cdot f_P \right] \right| \right)^2 > \varepsilon^2.
$$
Thus, each non-regular step increases the energy by more than $\varepsilon^2$. Since the energy is bounded above by $1$, the procedure stops after at most $t \le \lceil \varepsilon^{-2} \rceil$ steps. Applying the growth function at each step gives the bound $\mathcal{P}_t \in \mathsf{Part}_{\Phi^{\circ t}(M_0)}(X)$, with a final cardinality at most $\Phi^{\circ \lceil \varepsilon^{-2} \rceil}(M_0)$.
\end{proof}

\section{Deducing Green's arithmetic regularity lemma}\label{sec:green}
Fix a prime $p$ and let $X=\mathbb F_p^n$, equipped with the uniform probability
measure. For $\xi,x\in\mathbb F_p^n$, set
$
\langle \xi,x\rangle=\sum_{j=1}^n \xi_jx_j\in\mathbb F_p
$
and define
$\chi_\xi(x)=\omega^{\langle \xi,x\rangle}$, where $\omega=e^{2\pi i/p}$.
Let $A\subseteq \mathbb F_p^n$, and apply
\Cref{thm:abstract-regularity} to the function $g=\mathbf 1_A$.
\medskip

If $H\le \mathbb F_p^n$ is a subspace, write $\mathcal P_H$ for the partition of
$\mathbb F_p^n$ into cosets of $H$, and write
$$
H^\perp =\{\xi\in\mathbb F_p^n:\langle \xi,h\rangle=0
\text{ for every }h\in H\}.
$$
For a coset $P=a+H$, we identify $P$ with $H$ by translation. We then take as
local test functions
$$
\mathcal F_{a+H}
=
\{0\}
\cup
\left\{
x\mapsto \overline{\chi_\xi(x-a)}:
\xi\in\mathbb F_p^n\setminus H^\perp
\right\}.
$$
The irregularity of $A$ on the coset $a+H$ is exactly the largest
non-trivial Fourier coefficient of the balanced restriction of $\mathbf 1_A$ to
that coset
$$
\operatorname{irr}_{\mathbf 1_A}(a+H)
=
\max_{\xi\in\mathbb F_p^n\setminus H^\perp}
\left|
\mathop{\mathbb E}\limits_{x\in a+H}
\left[
\left(\mathbf 1_A(x)-\mathop{\mathbb E}\limits_{y\in a+H}\mathbf 1_A(y)\right)
\overline{\chi_\xi(x-a)}
\right]
\right|,
$$
with the convention that this maximum is $0$ when $H=\{0\}$.

The refinement scheme is the usual Fourier refinement. Suppose that, for each
coset $a+H$, we have chosen a witness frequency
$\xi_{a+H}\in\mathbb F_p^n\setminus H^\perp$, where the zero witness imposes no
condition. Define
$$
H'
=
\left\{
h\in H:
\chi_{\xi_{a+H}}(h)=1
\text{ for every chosen non-zero witness }\xi_{a+H}
\right\}.
$$
Then $\mathcal P_{H'}$ refines $\mathcal P_H$. Moreover, every chosen witness is
constant on each coset of $H'$ inside its ambient coset $a+H$. Therefore, the measurability condition required in \Cref{thm:abstract-regularity} is satisfied.

It remains only to record the growth of the refinement. If
$\operatorname{codim} H=r$, then $\mathcal P_H$ has $p^r$ parts, and the above
construction adds at most one independent linear condition for each coset of
$H$. Hence,
$
\operatorname{codim} H'\le r+p^r.
$
Equivalently, if $\mathcal P_H$ has at most $M$ parts, then
$\mathcal P_{H'}$ has at most
$p^{\lceil \log_p M\rceil+M}$ parts. This gives an admissible refinement
scheme with a growth function depending only on $p$.

We now recover the usual form of Green's arithmetic regularity lemma.

\begin{cor}[Green's arithmetic regularity lemma {\cite{Green2005}}]
\label{cor:green-usual}
For every prime $p$ and every $\varepsilon>0$, there is an integer
$M=M(p,\varepsilon)$ such that the following holds. For every $n$ and every
$A\subseteq\mathbb F_p^n$, there is a subspace
$H\le\mathbb F_p^n$ of codimension at most $M$ such that all but at most an
$\varepsilon$-fraction of the cosets $a+H$ satisfy
$$
\max_{\xi\in\mathbb F_p^n\setminus H^\perp}
\left|
\mathop{\mathbb E}\limits_{x\in a+H}
\left[
\left(\mathbf 1_A(x)-\mathop{\mathbb E}\limits_{y\in a+H}\mathbf 1_A(y)\right)
\overline{\chi_\xi(x-a)}
\right]
\right|
\le \varepsilon .
$$
\end{cor}

\begin{proof}
Start from the trivial coset partition $\mathcal P_{\mathbb F_p^n}$. The
refinement scheme described above stays inside the class of coset partitions.
Hence, \Cref{thm:abstract-regularity} yields a subspace
$H\le \mathbb F_p^n$, of codimension bounded only in terms of $p$ and $\eta$,
such that the partition $\mathcal P_H$ is $\eta$-regular for
$\mathbf 1_A$. Unwinding the definition of regularity, we obtain the following
result. For every prime $p$ and every $\eta>0$, there is an integer
$M=M(p,\eta)$ such that the following holds. For every $n$ and every
$A\subseteq \mathbb F_p^n$, there is a subspace $H\le \mathbb F_p^n$ of
codimension at most $M$ such that
$$
\mathop{\mathbb E}\limits_{a\in \mathbb F_p^n/H}
\max_{\xi\in\mathbb F_p^n\setminus H^\perp}
\left|
\mathop{\mathbb E}\limits_{x\in a+H}
\left[
\left(\mathbf 1_A(x)-\mathop{\mathbb E}\limits_{y\in a+H}\mathbf 1_A(y)\right)
\overline{\chi_\xi(x-a)}
\right]
\right|
\le \eta .
$$
Taking $\eta=\varepsilon^2$, it follows that
$$
\mathop{\mathbb E}\limits_{a\in \mathbb F_p^n/H}
\operatorname{irr}_{\mathbf 1_A}(a+H)
\le \varepsilon^2.
$$
By Markov's inequality, the fraction of cosets $a+H$ for which
$\operatorname{irr}_{\mathbf 1_A}(a+H)>\varepsilon$ is at most $\varepsilon$.
This is exactly the desired conclusion.
\end{proof}

\section{Deducing Szemerédi's graph regularity lemma}\label{sec: Szemeredi}
Let $G=(V,E)$ be a finite simple graph and write $n=|V|$. We work on
$X=V\times V$ with the uniform probability measure, and we apply
Theorem \ref{thm:abstract-regularity} to the function $g(x,y)=\1_E(x,y)$, where edges are counted as ordered pairs. Thus $g(x,y)=1$ if $xy\in E$,
and $g(x,y)=0$ otherwise. The diagonal is irrelevant for the argument. If $\mathcal V=\{V_1,\dots,V_m\}$ is a partition of $V$, let $\mathcal V^2=\{V_i\times V_j:1\le i,j\le m\}$ be the induced product partition of $V\times V$. For non-empty subsets
$A,B\subseteq V$, define $e(A,B)=\abs*{\{(a,b)\in A\times B:ab\in E\}}$, and $d(A,B)=\frac{e(A,B)}{|A||B|}$. 
On the part $V_i\times V_j$, we take as local test functions all rectangle
indicators:
$$
\mathcal F_{V_i\times V_j}
=
\{\1_{S\times T}:S\subseteq V_i,\ T\subseteq V_j\}.
$$
For this choice, the irregularity on  $V_i\times V_j$ is
$$
\operatorname{irr}_g(V_i\times V_j)
=
\frac{1}{|V_i||V_j|}
\max_{S\subseteq V_i,\ T\subseteq V_j}
\abs*{
e(S,T)-d(V_i,V_j)|S||T|
}.
$$
Indeed,
$$
\E_{V_i\times V_j}
\left[
(g-d(V_i,V_j))\1_{S\times T}
\right]
=
\frac{
e(S,T)-d(V_i,V_j)|S||T|
}{
|V_i||V_j|
}.
$$
Therefore the $\eta$-regularity condition for the product partition
$\mathcal V^2$ becomes
$$
\sum_{i,j}
\max_{S\subseteq V_i,\ T\subseteq V_j}
\abs*{
e(S,T)-d(V_i,V_j)|S||T|
}
\le
\eta n^2.
$$

The refinement scheme is the usual common refinement. Suppose that, for every
ordered pair $(i,j)$, a witness rectangle $S_{ij}\times T_{ij}\subseteq
V_i\times V_j$ has been chosen. We refine each vertex class $V_i$ by all
sets $S_{ij}$ with $1\le j\le m$, and $T_{ji}$ with $1\le j\le m$. 
Equivalently, we replace $V_i$ by the parts of the Boolean algebra generated
inside $V_i$ by these at most $2m$ subsets. The resulting vertex partition
$\mathcal V'$ refines $\mathcal V$, and the product partition
$(\mathcal V')^2$ refines $\mathcal V^2$. Moreover, every chosen rectangle
indicator $\1_{S_{ij}\times T_{ij}}$ is constant on each part of
$(\mathcal V')^2|_{V_i\times V_j}$, exactly as required by the abstract
refinement scheme.

If $\mathcal V$ has $m$ parts, then each $V_i$ is split into at most
$2^{2m}$ parts. Hence $\mathcal V'$ has at most $m2^{2m}$ parts, and the product partition has at most $m^2 2^{4m}$ parts. Thus, this is an admissible refinement scheme with a growth function depending only on the
current number of parts.

We now recover the usual form of Szemerédi's graph regularity lemma.

\begin{cor}[Szemerédi's graph regularity lemma {\cite{Szemeredi1978}}]
\label{cor:szemeredi-weighted}
For every $\eps>0$, there is an integer $M=M(\eps)$ such that every finite
graph $G=(V,E)$ admits a partition
$V=V_1\cup\cdots\cup V_m$
with $1\le m\le M$ such that
$$
\sum_{\substack{1\le i,j\le m\\ (V_i,V_j)\text{ is not }\eps\text{-regular}}}
|V_i||V_j|
\le
\eps |V|^2.
$$
\end{cor}

\begin{proof}
Start from the trivial partition $V_1=V$ and apply
\Cref{thm:abstract-regularity} to the induced product partition of
$V\times V$, with the rectangle test families and refinement scheme described
above. The final product partition is induced by a vertex partition
$\mathcal V=\{V_1,\dots,V_m\}$, where $m$ is bounded only in terms of
$\eta$. Unwinding the definition of abstract regularity, we obtain the
following result. For every $\eta>0$, there is an integer $M=M(\eta)$ such
that every finite graph $G=(V,E)$ admits a partition
$V=V_1\cup\cdots\cup V_m$, $1\le m\le M$, satisfying
$$
\sum_{i,j}
\max_{S\subseteq V_i,\ T\subseteq V_j}
\abs*{
e(S,T)-d(V_i,V_j)|S||T|
}
\le
\eta |V|^2.
$$
We now relate this averaged formulation to the usual pairwise notion of
regularity. Recall that a pair $(A,B)$ is $\eps$-regular if, for all
$S\subseteq A$ and $T\subseteq B$ with
$|S|\ge \eps |A|$, $|T|\ge \eps |B|$, one has
$\abs*{d(S,T)-d(A,B)}\le \eps.$ If $(V_i,V_j)$ is not $\eps$-regular, then there are $S\subseteq V_i$ and $T\subseteq V_j$ such that
$|S|\ge \eps |V_i|$, $|T|\ge \eps |V_j|$, and $\abs*{d(S,T)-d(V_i,V_j)}>\eps$. 
Equivalently,
$$
\abs*{
e(S,T)-d(V_i,V_j)|S||T|
}
>
\eps |S||T|
\ge
\eps^3 |V_i||V_j|.
$$
Let $\eta=\eps^4$. Let $\mathcal B$ be the set of ordered pairs $(i,j)$ for which $(V_i,V_j)$ is not $\eps$-regular. We have that
$$
\eps^3
\sum_{(i,j)\in\mathcal B}
|V_i||V_j|
\le
\sum_{i,j}
\max_{S\subseteq V_i,\ T\subseteq V_j}
\abs*{
e(S,T)-d(V_i,V_j)|S||T|
}
\le
\eps^4 |V|^2.
$$
Dividing by $\eps^3$ and recalling that $\eta=\eps^4$, we obtain
$$
\sum_{\substack{1\le i,j\le m\\ (V_i,V_j)\text{ is not }\eps\text{-regular}}}
|V_i||V_j|
\le
\eps |V|^2.
$$
\end{proof}

\section{Deducing the regularity lemma for Boolean functions}\label{sec: Dinur-Friedgut}
For a finite set $I$ and $p\in(0,1)$, let $\mu_p^I$ denote the
$p$-biased product measure on $\{0,1\}^I$, that is, $$\mu_p^I(\{z\})=\prod_{i\in I}p^{z_i}(1-p)^{1-z_i}.$$ When the set of coordinates is clear, we simply write $\mu_p$. 

If $f:\{0,1\}^n\to[0,1]$, we write $\mu_p(f)=\E_{x\sim\mu_p^{[n]}}f(x)$. If $T\subseteq[n]$ and $z\in\{0,1\}^T$, then $f_{T\to z}$ denotes the
restriction of $f$ obtained by fixing the coordinates in $T$ according to
$z$.

Let $r\in\mathbb N$ and $\eps>0$. A function
$f:\{0,1\}^n\to[0,1]$ is called $(r,\eps)$-\emph{quasirandom with
respect to $p$} if, for every $R\subseteq[n]$ with $|R|\le r$ and every
$w\in\{0,1\}^R$,$$\abs*{\mu_p(f_{R\to w})-\mu_p(f)}\le \eps.$$
More generally, if $T\subseteq[n]$ and $z\in\{0,1\}^T$, we say that
$f_{T\to z}$ is $(r,\eps)$-quasirandom if, for every
$R\subseteq[n]\setminus T$ with $|R|\le r$ and every
$w\in\{0,1\}^R$, $$\abs*{\mu_p(f_{T\to z,R\to w})-\mu_p(f_{T\to z})}\le \eps.$$

Fix $p\in(0,1)$ and put $X=\{0,1\}^n$, $\mu=\mu_p^{[n]}$, $g=f$. For $T\subseteq[n]$, let $\mathcal P_T$ be the partition of $\{0,1\}^n$
into the parts $P_z=\{x\in\{0,1\}^n:x_T=z\}$ with $z\in\{0,1\}^T.$ On the part $P_z$ we take as test functions all indicators of subcubes
obtained by fixing at most $r$ further coordinates:
$$
\mathcal F_{P_z}
=
\{0\}
\cup
\Bigl\{
\1_{\{x_R=w\}}\bigm|_{P_z}:
R\subseteq[n]\setminus T,\ |R|\le r,\ w\in\{0,1\}^R
\Bigr\}.
$$
Here $\1_{\{x_R=w\}}\bigm|_{P_z}$ is viewed as a function on $P_z$. For such a test function we have
$$
\E_{P_z}
\Bigl[
\bigl(g-\E_{P_z}g\bigr)\1_{\{x_R=w\}}
\Bigr]
=
\mu_p^R(w)
\Bigl(
\mu_p(f_{T\to z,R\to w})-\mu_p(f_{T\to z})
\Bigr).
$$
Indeed, after conditioning on $x_T=z$, the remaining coordinates still
have the $p$-biased product distribution. Therefore
$$
\operatorname{irr}_g(P_z)
=
\sup_{\substack{R\subseteq[n]\setminus T\\ |R|\le r}}
\ \sup_{w\in\{0,1\}^R}
\mu_p^R(w)
\abs*{
\mu_p(f_{T\to z,R\to w})-\mu_p(f_{T\to z})
}.
$$

The refinement scheme is also immediate. Suppose that, for each part
$P_z$, a witness has been chosen. If the witness is non-zero, it is of
the form $\1_{\{x_{R_z}=w_z\}}$ with $R_z\subseteq[n]\setminus T$ and
$|R_z|\le r$. Define
$$
T'
=
T\cup\bigcup_{z\in\{0,1\}^T}R_z,
$$
where zero witnesses contribute nothing. Then $\mathcal P_{T'}$ refines
$\mathcal P_T$, and every chosen witness is constant on every part of
$\mathcal P_{T'}$ contained in the corresponding part of $\mathcal P_T$.
Moreover, $|T'|\le |T|+r2^{|T|}$. If one parametrizes admissible partitions by their number of parts, this
gives the growth bound
$\Phi_r(M)=\left\lceil M2^{rM}\right\rceil,
$
because a partition fixing at most $\lfloor\log_2 M\rfloor$ coordinates
has at most $M$ parts, and after refinement it has at most
$M2^{rM}$ parts.

\begin{cor}[Regularity lemma for Boolean functions]\label{cor:boolean-regularity}
For every $r\in\mathbb N$ and every $\eps,\zeta,\delta>0$, there is an
integer $J=J(r,\eps,\zeta,\delta)$ such that the following holds. Let
$\zeta<p<\frac12-\zeta$, and let
$f:\{0,1\}^n\to\{0,1\}$. Then there is a set of coordinates
$T\subseteq[n]$ with $|T|\le J$ such that
$$
\Pr_{z\sim\mu_p^T}
\left[
f_{T\to z}\text{ is not }(r,\eps)\text{-quasirandom with respect to }p
\right]
\le \delta.
$$
\end{cor}

\begin{proof}
Apply \cref{thm:abstract-regularity} to the partitions
$\mathcal P_T$ and to the test families described above, with regularity
parameter $\eta=\delta\,\zeta^r\eps$. This gives a coordinate set $T$ of size at most a constant $J=J(r,\eps,\zeta,\delta)$, independent of $n$, such that $\sum_{z\in\{0,1\}^T}
\mu_p^T(z)\operatorname{irr}_f(P_z)
\le\eta$. 

Let $B= \{z\in\{0,1\}^T: f_{T\to z}\text{ is not }(r,\eps)\text{-quasirandom}\}$. We claim that every $z\in B$ satisfies
$\operatorname{irr}_f(P_z)>\zeta^r\eps$. Indeed, if $z\in B$, then there exist
$R\subseteq[n]\setminus T$ with $|R|\le r$ and
$w\in\{0,1\}^R$ such that
$$\abs*{
\mu_p(f_{T\to z,R\to w})-\mu_p(f_{T\to z})
}
>
\eps.$$ 
Since $\zeta<p<\frac12-\zeta$, every part of the $p$-biased measure on
$R$ has measure at least $\zeta^{|R|}\ge \zeta^r$. Hence
$$
\operatorname{irr}_f(P_z)
\ge
\mu_p^R(w)
\abs*{
\mu_p(f_{T\to z,R\to w})-\mu_p(f_{T\to z})
}
>
\zeta^r\eps.
$$

Therefore
$$
\delta\zeta^r\eps
=
\eta
\ge
\sum_{z\in\{0,1\}^T}
\mu_p^T(z)\operatorname{irr}_f(P_z)
\ge
\sum_{z\in B}
\mu_p^T(z)\operatorname{irr}_f(P_z)
>
\mu_p^T(B)\zeta^r\eps.
$$
Cancelling $\zeta^r\eps$ gives $\mu_p^T(B)\le\delta$, which is the desired
conclusion.
\end{proof}
\printbibliography
\appendix
\section{Deducing the Bollobás--Nikiforov abstract regularity lemma}
\label{sec: comparisons}
We first recall the setup.  Let $(X,\mathcal A,\mu,\mathcal S)$ be an
SR-system in the sense of Bollobás and Nikiforov: $(X,\mathcal A,\mu)$ is a
normalized measure triple and $\mathcal S\subseteq\mathcal A$ is a semiring. Further, assume that there is an integer $r\ge 1$ such that, whenever
$S,T\in\mathcal S$, the difference $S\setminus T$ can be written as a
disjoint union of at most $r$ members of $\mathcal S$.  If $A,V\in\mathcal A$
and $\mu(V)>0$, write
$
d(A,V)=\frac{\mu(A\cap V)}{\mu(V)}.
$

Following Bollobás and Nikiforov, if $V\in\mathcal S$ has positive measure,
we say that $A$ is $\eps$-\emph{regular in} $V$ if
$$
\abs*{d(A,U)-d(A,V)}<\eps,
$$
for every $U\in\mathcal S$ with $U\subseteq V$ and
$\mu(U)>\eps\mu(V)$.  If $\mathcal P$ is a finite partition of $X$ into
members of $\mathcal S$, we say that $A$ is $\eps$-\emph{regular in}
$\mathcal P$ if
$$
\sum_{\substack{P\in\mathcal P\\ A\text{ is not }\eps\text{-regular in }P}}
\mu(P)<\eps .
$$

We shall also use the notion of a bounding family of partitions.  Let
$\Pi(\mathcal S)$ denote the set of finite partitions of $X$ into members of
$\mathcal S$.  A family $\Phi\subseteq \Pi(\mathcal S)$ \emph{bounds}
$\Pi(\mathcal S)$ if there is an increasing function
$\varphi:\mathbb N\to\mathbb N$ such that, for every
$\mathcal P\in\Pi(\mathcal S)$, there exists
$\mathcal Q\in\Phi$ refining $\mathcal P$ with
$|\mathcal Q|\le \varphi(|\mathcal P|)$.

We now recover the abstract regularity lemma of Bollobás and Nikiforov.

\begin{cor}[Bollobás--Nikiforov abstract regularity lemma]
Let $(X,\mathcal A,\mu,\mathcal S)$ be an SR-system, and let
$\Phi\subseteq\Pi(\mathcal S)$ be a family of partitions bounding
$\Pi(\mathcal S)$ with rate $\varphi$.  Let
$\mathcal L\subseteq\mathcal A$ be a finite family of measurable sets, let
$\mathcal P\in\Pi(\mathcal S)$, and let $\eps>0$.  Then there is an integer
$q=q(\eps,|\mathcal L|,|\mathcal P|)$ and a partition
$\mathcal Q\in\Phi$ such that
$\mathcal Q$ refines $\mathcal P$, $|\mathcal Q|\le q,
$ and every $A\in\mathcal L$ is $\eps$-regular in $\mathcal Q$.
\end{cor}

\begin{proof}
Let $\ell=|\mathcal L|$.  If $\ell=0$ there is nothing to prove, so assume
$\ell\ge 1$.  Since $\Phi$ bounds $\Pi(\mathcal S)$, choose first a partition
$\mathcal P_0\in\Phi$ refining $\mathcal P$ with $|\mathcal P_0|\le \varphi(|\mathcal P|)$. We apply the proof of \Cref{thm:abstract-regularity} simultaneously to the
functions $\1_A$, $A\in\mathcal L$.  For a part $P$ of a partition
$\mathcal R\in\Phi$, take as local test functions
$$
\mathcal F_P
=
\{0\}\cup
\{\1_U|_P: U\in\mathcal S,\ U\subseteq P\}.
$$
For $A\in\mathcal L$ and $P\in\mathcal R$, the corresponding local
irregularity is
$$
\operatorname{irr}_A(P)
=
\sup_{\substack{U\in\mathcal S\\ U\subseteq P}}
\left|
\E_P\left[
\bigl(\1_A-d(A,P)\bigr)\1_U
\right]
\right|
=
\sup_{\substack{U\in\mathcal S\\ U\subseteq P}}
\frac{\mu(U)}{\mu(P)}
\abs*{d(A,U)-d(A,P)}.
$$

The only point to check is that the refinement scheme required by
\Cref{thm:abstract-regularity} is available inside the family $\Phi$.  Suppose
that $\mathcal R\in\Phi$ has $m$ parts.  For each pair
$(A,P)\in\mathcal L\times\mathcal R$, choose a witness
$U_{A,P}\in\mathcal S$ with $U_{A,P}\subseteq P$.  Inside a fixed part $P$,
we refine by all the sets $U_{A,P}$, $A\in\mathcal L$.  Since
$\mathcal S$ is $r$-built, refining one member of $\mathcal S$ by one
semiring subset splits it into at most $r+1$ members of $\mathcal S$.
Iterating this for the $\ell$ witnesses inside $P$, we obtain a partition of
$P$ into at most $(r+1)^\ell$ members of $\mathcal S$ on which all the
indicators $\1_{U_{A,P}}$ are constant.  Doing this for every
$P\in\mathcal R$ gives a partition $\mathcal R^\ast\in\Pi(\mathcal S)$ with
$
|\mathcal R^\ast|\le m(r+1)^\ell $.
Now use the bounding property of $\Phi$ to choose
$\mathcal R'\in\Phi$ refining $\mathcal R^\ast$ with
$$
|\mathcal R'|\le \varphi\bigl(m(r+1)^\ell\bigr).
$$
Thus we have an admissible refinement scheme with growth function
$$
\Psi(m)=\varphi\bigl(m(r+1)^\ell\bigr).
$$
Run the energy-increment argument with the total energy
$$
\mathcal E_{\mathcal L}(\mathcal R)
=
\sum_{A\in\mathcal L}
\left\|
\E[\1_A\mid \mathcal R]
\right\|_{L^2(\mu)}^2.
$$
This energy lies between $0$ and $\ell$.  If
$$
\sum_{A\in\mathcal L}
\sum_{P\in\mathcal R}
\mu(P)\operatorname{irr}_A(P)>\eta,
$$
then choose witnesses $U_{A,P}$ and refine as above.  The same computation as
in the proof of \Cref{thm:abstract-regularity} gives
$$
\mathcal E_{\mathcal L}(\mathcal R')
-
\mathcal E_{\mathcal L}(\mathcal R)
\ge
\sum_{A\in\mathcal L}
\sum_{P\in\mathcal R}
\mu(P)
\left|
\E_P\left[
\bigl(\1_A-d(A,P)\bigr)\1_{U_{A,P}}
\right]
\right|^2 .
$$
Since the weights $\mu(P)$ sum to $1$ for each fixed $A$, and there are
$\ell$ choices of $A$, Jensen's inequality gives
$$
\sum_{A\in\mathcal L}
\sum_{P\in\mathcal R}
\mu(P)
\left|
\E_P\left[
\bigl(\1_A-d(A,P)\bigr)\1_{U_{A,P}}
\right]
\right|^2
\ge
\frac{\eta^2}{\ell}.
$$
Thus every non-regular step increases the total energy by at least
$\eta^2/\ell$.  Since the total energy is at most $\ell$, the procedure
stops after at most
$
\left\lceil \frac{\ell^2}{\eta^2}\right\rceil
$
steps.  We obtain a partition $\mathcal Q\in\Phi$ refining $\mathcal P_0$,
hence refining $\mathcal P$, with
$
|\mathcal Q|
\le
\Psi^{\circ \left\lceil \frac{\ell^2}{\eta^2}\right\rceil}\bigl(\varphi(|\mathcal P|)\bigr),
$
and satisfying
$$
\sum_{A\in\mathcal L}
\sum_{Q\in\mathcal Q}
\mu(Q)\operatorname{irr}_A(Q)
\le \eta .
$$
It remains to translate this averaged test-function regularity into the
Bollobás--Nikiforov notion.  Fix $A\in\mathcal L$, and let
$\mathcal B_A$ be the set of parts $Q\in\mathcal Q$ in which $A$ is not
$\eps$-regular.  For each $Q\in\mathcal B_A$, there is a set
$U\in\mathcal S$, $U\subseteq Q$, such that
$\mu(U)>\eps\mu(Q)$ and $\abs*{d(A,U)-d(A,Q)}>\eps$. 
Therefore
$$
\operatorname{irr}_A(Q)
\ge
\frac{\mu(U)}{\mu(Q)}
\abs*{d(A,U)-d(A,Q)}
>
\eps^2 .
$$
Hence
$$
\eps^2
\sum_{Q\in\mathcal B_A}\mu(Q)
<
\sum_{Q\in\mathcal B_A}
\mu(Q)\operatorname{irr}_A(Q)
\le
\sum_{Q\in\mathcal Q}
\mu(Q)\operatorname{irr}_A(Q)
\le
\eta .
$$
Taking $\eta=\eps^3$ gives
$
\sum_{Q\in\mathcal B_A}\mu(Q)<\eps$. Thus every $A\in\mathcal L$ is $\eps$-regular in $\mathcal Q$. Finally, with $\eta=\eps^3$, one may take
$
q(\eps,\ell,|\mathcal P|)
=
\Psi^{\circ \left\lceil \ell^2\eps^{-6}\right\rceil}
\bigl(\varphi(|\mathcal P|)\bigr)$, and 
$\Psi(m)=\varphi\bigl(m(r+1)^\ell\bigr)$. 
\end{proof}

\section{Deducing Tao's probabilistic regularity lemma}
\label{sec: tao}
We shall use the following harmless variant of \Cref{thm:abstract-regularity}:
the function $g:X\to[0,1]$ may be replaced by a real-valued function
$X\in L^2(\mu)$ with $\|X\|_2\le 1$.  The proof is unchanged, since the
only point at which boundedness is used is the energy bound
$$0\le \|\E[X\mid \mathcal P]\|_2^2\le \|X\|_2^2\le 1.$$

We identify a finite sub-$\sigma$-algebra with its finite partition into
parts.  If $\mathcal B$ is a finite sub-$\sigma$-algebra, write
$\operatorname{complex}(\mathcal B)$ for the least number of events needed
to generate it.  Thus $\mathcal B$ has at most
$2^{\operatorname{complex}(\mathcal B)}$ parts.  If
$(\mathcal B_i)_{i\in I}$ is a finite family of finite sub-$\sigma$-algebras,
write
$
\bigvee_{i\in I}\mathcal B_i
$
for their common refinement. We first isolate the single-level regularization statement supplied by
\Cref{thm:abstract-regularity}.

\begin{lem}
\label{lem:tao-product-regularization}
Let $(\Omega,\mathcal B_{\max},\mu)$ be a probability space, let
$(\mathcal B_{i,\max})_{i\in I}$ be a finite family of sub-$\sigma$-algebras
of $\mathcal B_{\max}$, and let $X\in L^2(\mathcal B_{\max})$ satisfy
$\|X\|_2\le 1$.  Let $(\mathcal B_i)_{i\in I}$ be finite sub-$\sigma$-algebras
with
$
\mathcal B_i\subseteq \mathcal B_{i,\max}$ and $\operatorname{complex}(\mathcal B_i)\le M
$
for all $i\in I$.  Then, for every $\delta>0$, there are finite
sub-$\sigma$-algebras
$
\mathcal B_i\subseteq \mathcal B_i'\subseteq \mathcal B_{i,\max}
$
such that
$$
\left|
\E\left[
\left(
X-\E\left[X\mid \bigvee_{i\in I}\mathcal B_i'\right]
\right)
\prod_{i\in I}1_{A_i}
\right]
\right|
\le \delta
$$
for every choice of events $A_i\in\mathcal B_{i,\max}$. Moreover, the complexities of the $\mathcal B_i'$ are bounded in terms of
$M,\delta$, and $|I|$ only.
\end{lem}

\begin{proof}
Let
$
\mathcal P=\bigvee_{i\in I}\mathcal B_i
$
be the joint part partition. For a part $P\in\mathcal P$, take as local
test functions
$
\mathcal F_P
=
\left\{
\left.
\prod_{i\in I}1_{A_i}
\right|_P
\,:\,
A_i\in\mathcal B_{i,\max}
\right\}.
$
These functions are bounded by $1$.

We now describe the refinement scheme.  Suppose that, for each part
$P\in\mathcal P$, a witness
$
f_P=
\left.
\prod_{i\in I}1_{A_i(P)}
\right|_P
$
has been chosen, with $A_i(P)\in\mathcal B_{i,\max}$.  Refine each
$\mathcal B_i$ by adjoining all the events $A_i(P)$, as $P$ ranges over
the parts of $\mathcal P$.  Denote the resulting algebra by
$\mathcal B_i^+$.  Then
$
\mathcal B_i\subseteq \mathcal B_i^+\subseteq \mathcal B_{i,\max}.
$
Moreover, every $f_P$ is measurable with respect to
$\bigvee_{i\in I}\mathcal B_i^+$ inside $P$, so the measurability condition
in \Cref{thm:abstract-regularity} is satisfied.

If $\operatorname{complex}(\mathcal B_i)\le M$ for every $i$, then
$\bigvee_{i\in I}\mathcal B_i$ has at most $2^{|I|M}$ parts.  Hence each
$\mathcal B_i$ is refined by adjoining at most $2^{|I|M}$ new events, and
so
$
\operatorname{complex}(\mathcal B_i^+)
\le
M+2^{|I|M}.
$
Thus we have an admissible refinement scheme with a growth function depending
only on $|I|$.

Applying \Cref{thm:abstract-regularity} with parameter $\delta$, starting
from $\bigvee_i\mathcal B_i$, gives finite refinements
$\mathcal B_i'\supseteq \mathcal B_i$ such that, if
$
\mathcal P'=\bigvee_{i\in I}\mathcal B_i',
$
then
$
\sum_{P\in\mathcal P'}
\mu(P)
\sup_{A_i\in\mathcal B_{i,\max}}
\left|
\E_P\left[
\left(X-\E_P X\right)
\prod_{i\in I}1_{A_i}
\right]
\right|
\le \delta.
$
Now fix arbitrary $A_i\in\mathcal B_{i,\max}$.  Since
$\E[X\mid\mathcal P']$ is equal to $\E_PX$ on each part $P\in\mathcal P'$,
we have
$$
\begin{aligned}
&
\left|
\E\left[
\left(
X-\E[X\mid\mathcal P']
\right)
\prod_{i\in I}1_{A_i}
\right]
\right| \le
\sum_{P\in\mathcal P'}
\mu(P)
\left|
\E_P\left[
\left(X-\E_PX\right)
\prod_{i\in I}1_{A_i}
\right]
\right|
\le \delta.
\end{aligned}
$$
This proves the desired product-test regularity.  The complexity bound follows
by iterating the growth function for at most $\lceil\delta^{-2}\rceil$ steps.
\end{proof}
We now recover Tao's theorem.

\begin{cor}
\label{cor:tao-probabilistic}
Let $(\Omega,\mathcal B_{\max},\mu)$ be a probability space, let
$(\mathcal B_{i,\max})_{i\in I}$ be a finite family of sub-$\sigma$-algebras
of $\mathcal B_{\max}$, and let $X\in L^2(\mathcal B_{\max})$ satisfy
$\|X\|_2\le 1$.  Let $\eps>0$, let $m\ge 0$, and let
$F:\mathbb R_+\to\mathbb R_+$ be increasing.  Then there are finite
sub-$\sigma$-algebras
$\mathcal B_i\subseteq\mathcal B_i'\subseteq\mathcal B_{i,\max}$, $i\in I$, 
and a number $M=O_{\eps,F,m,|I|}(1)$, such that
$M\ge m$, $\operatorname{complex}(\mathcal B_i)\le M$
for every $i\in I$,
$
\left\|
\E\left[X\mid \bigvee_{i\in I}\mathcal B_i'\right]
-
\E\left[X\mid \bigvee_{i\in I}\mathcal B_i\right]
\right\|_2
\le \eps,
$
and, for every choice of events $A_i\in\mathcal B_{i,\max}$,
$$
\left|
\E\left[
\left(
X-\E\left[X\mid \bigvee_{i\in I}\mathcal B_i'\right]
\right)
\prod_{i\in I}1_{A_i}
\right]
\right|
\le \frac{1}{F(M)}.
$$
\end{cor}

\begin{proof}
We run a coarse/fine energy-increment procedure.  Begin for every $i\in I$, with the trivial
algebras
$\mathcal B_i^{(0)}=\{\emptyset,\Omega\}$.
Suppose that the current coarse algebras
$(\mathcal B_i^{(s)})_{i\in I}$ have been constructed, and put
$
M_s=
\max\left\{
m,\max_{i\in I}\operatorname{complex}(\mathcal B_i^{(s)})
\right\}.
$
Apply \Cref{lem:tao-product-regularization} to the coarse algebras
$(\mathcal B_i^{(s)})_{i\in I}$ with
$
\delta_s=\frac{1}{F(M_s)}.
$
This gives refinements
$
\mathcal B_i^{(s)}\subseteq \mathcal B_i^{\mathrm{fine}}
\subseteq \mathcal B_{i,\max}
$
such that, for every $A_i\in\mathcal B_{i,\max}$,
$$
\left|
\E\left[
\left(
X-\E\left[X\mid \bigvee_{i\in I}\mathcal B_i^{\mathrm{fine}}\right]
\right)
\prod_{i\in I}1_{A_i}
\right]
\right|
\le \frac{1}{F(M_s)}.
$$

If
$
\left\|
\E\left[X\mid \bigvee_{i\in I}\mathcal B_i^{\mathrm{fine}}\right]
-
\E\left[X\mid \bigvee_{i\in I}\mathcal B_i^{(s)}\right]
\right\|_2
\le \eps,
$
we stop and set
$
\mathcal B_i=\mathcal B_i^{(s)},
\mathcal B_i'=\mathcal B_i^{\mathrm{fine}},
M=M_s.
$
The required conclusions then hold.

Otherwise, set
$
\mathcal B_i^{(s+1)}=\mathcal B_i^{\mathrm{fine}}
$
and continue.  We claim that this can happen at most $\lceil\eps^{-2}\rceil$
times.  Indeed, write
$
\mathcal C_s=\bigvee_{i\in I}\mathcal B_i^{(s)}$ and $
\mathcal C_s^{\mathrm{fine}}
=
\bigvee_{i\in I}\mathcal B_i^{\mathrm{fine}}.
$
Since $\mathcal C_s^{\mathrm{fine}}$ refines $\mathcal C_s$, orthogonality
of conditional expectation gives
$
\left\|\E[X\mid\mathcal C_s^{\mathrm{fine}}]\right\|_2^2
-
\left\|\E[X\mid\mathcal C_s]\right\|_2^2
=
\left\|
\E[X\mid\mathcal C_s^{\mathrm{fine}}]
-
\E[X\mid\mathcal C_s]
\right\|_2^2.
$
Thus every non-stopping step increases the energy by more than $\eps^2$.
But the energy is always at most $\|X\|_2^2\le 1$, so there are at most
$\lceil\eps^{-2}\rceil$ non-stopping steps.

It remains only to bound the final value of $M$.  By
\Cref{lem:tao-product-regularization}, the complexity of the fine algebras
constructed from coarse complexity at most $M_s$ is bounded by a quantity
depending only on $M_s$, $F(M_s)$, and $|I|$.  Hence there is an increasing
function $\Theta=\Theta_{F,|I|}$ such that
$
M_{s+1}\le \Theta(M_s)
$
whenever the procedure does not stop.  Since the procedure has at most
$\lceil\eps^{-2}\rceil$ non-stopping steps and $M_0=m$, the final value
satisfies
$
M\le \Theta^{\circ \lceil\eps^{-2}\rceil}(m).
$
Therefore $M=O_{\eps,F,m,|I|}(1)$, as required.
\end{proof}
\end{document}